\documentclass[11pt]{article}
\usepackage{amsmath,amssymb,amsthm}
\usepackage[margin=1in]{geometry}
\usepackage{mathtools}
\usepackage{xcolor}
\usepackage{url}
\numberwithin{equation}{section}

\newtheorem{theorem}{Theorem}[section]
\newtheorem{proposition}[theorem]{Proposition}
\newtheorem{lemma}[theorem]{Lemma}
\newtheorem{corollary}[theorem]{Corollary}
\newtheorem{assumption}[theorem]{Assumption}
\theoremstyle{remark}
\newtheorem{remark}[theorem]{Remark}

\newcommand{\Bn}{\mathbb B_n}
\newcommand{\D}{\mathbb D}
\newcommand{\C}{\mathbb C}
\newcommand{\R}{\mathbb R}
\newcommand{\dd}{\,\mathrm d}
\newcommand{\norm}[1]{\left\|#1\right\|}

\title{The Reciprocal Problem on Weighted Bergman Spaces}
\author{%
Guangfu Cao, \hskip5mm Li He$^1$, \hskip5mm Shuqing Zhang$^1$\\
School of Mathematics and Information Science,\\
Guangzhou University,\\
Guangzhou 510006, China\\
\\[1ex]
}
\date{}

\begin{document}
\maketitle
\footnotetext[1]{Corresponding authors.}

\begin{abstract}
We study the reciprocal problem for weighted Bergman spaces: if $f\in A_\alpha^p(\Bn)$ and $\inf_{\Bn}|f|>0$, does it follow that $1/f\in A_\alpha^p(\Bn)$? We determine the range of parameters for which the answer is affirmative. In particular, we prove that functions in $A_{\alpha}^p(\D)$ have the reciprocal property for all $\alpha \in \R$ and $p\geq 1,$ where $\D$ denotes the unit disk in $\C.$ Moreover, functions in $A_{\alpha}^p(\Bn)$ have the reciprocal property for   all $\alpha \in \R$ when $n=2$ and $p=2.$ In addition, we resolve the reciprocal problem in the three-dimensional Drury--Arveson space $H_3^2$ and obtain an equivalent condition in the four-dimensional space $H_4^2$. We also settle the reciprocal problem for general $A_{\alpha}^p(\Bn)$ under some  additional conditions. 
\end{abstract}

\section{Introduction}

Let
\[
\Bn=\{z=(z_1,\ldots,z_n)\in\C^n:|z|<1\}
\]
denote the unit ball, $\D=\Bn$ for $n=1$, and let $R$ denote the radial derivative:
\[
Rf(z)=\sum_{j=1}^n z_j\frac{\partial f}{\partial z_j}(z).
\]
If $f=\sum_{k\geq0}f_k$ is its homogeneous expansion, then $Rf=\sum_{k\geq0}k f_k$. In general, for $\beta\in\mathbb R$ one can define the fractional radial operator
\[
R^\beta f=\sum_{k\geq1}k^\beta f_k.
\]
For example, the Hardy--Sobolev space $H^2_\beta$ consists of all holomorphic functions satisfying $R^\beta f\in H^2$, with norm $\|f\|_{H^2_\beta}^2=|f(0)|^2+\|R^\beta f\|_{H^2}^2$. In multiplier theory, when the corresponding multiplication operator is bounded, one has the spectral characterization
\[
\sigma(M_\varphi;H^2_\beta)=\overline{\varphi(\Bn)},
\]
and the spectral theory of multipliers on Hardy--Sobolev spaces can be found in \cite{CaoHeZhu2018}.
For $0<p<\infty$ and $\alpha\in\R$, choose an integer $m\geq0$ such that
\[
pm+\alpha>-1.
\]
Set
\[
\gamma=\alpha+pm,\qquad
d\mu_\gamma(z)=(1-|z|^2)^\gamma\dd v(z).
\]
Let $\dd v_\alpha(z)=(1-|z|^2)^\alpha\dd v(z)$ denote the original weighted measure (which need not be finite when $\alpha\leq-1$). Apart from the constant term, the weighted Bergman space $A_\alpha^p(\Bn)$ can be characterized by
\begin{equation}
f\in A_\alpha^p(\Bn)
\quad\Longleftrightarrow\quad
R^mf\in L^p(\Bn,d\mu_\gamma).
\label{eq:derivative-characterization}
\end{equation}
This derivative characterization can be found in \cite{ZhaoZhu2008,Zhu2005}, showing that Hardy--Sobolev spaces and weighted Bergman spaces are essentially equivalent. Consequently, there is the corresponding result for weighted Bergman spaces:
\[
\sigma(M_\varphi;A^2_\alpha)=\overline{\varphi(\Bn)}.
\]
\cite{Zhu2016} posed the following reciprocal problem: if
\begin{equation}
 f\in A_\alpha^p(\Bn),\qquad
 \inf_{z\in\Bn}|f(z)|>0,
\end{equation}
does one have
 \[
 \frac1f\in A_\alpha^p(\Bn)?
 \]
 Zhu pointed out  that  even for the unit disk,  this problem is still open for general $p$ and $\alpha$!
 
Equation~\eqref{eq:derivative-characterization} reformulates the problem as an integrability criterion for higher-order derivatives. The case $\alpha>-1$ is immediate, since the weighted measure is then finite; the difficulty arises when $\alpha\leq-1$. A complete treatment for all $\alpha\in\R$ and $0<p<\infty$ remains open and is tied to endpoint embedding problems for analytic Besov spaces.

\section{Higher-Order Derivative Formula}
\begin{lemma}\label{lem:bell}
For any integer $m\geq1$ and any zero-free holomorphic function $f$, there exist signed constants $C_{\ell_1,\ldots,\ell_r}$ such that
\begin{equation}
 R^m\left(\frac1f\right)
 =
 \sum_{r=1}^{m}\frac{1}{f^{r+1}}
 \sum_{\ell_1+\cdots+\ell_r=m}
 C_{\ell_1,\ldots,\ell_r}
 R^{\ell_1}f\cdots R^{\ell_r}f .
\label{eq:bell-expansion}
\end{equation}
Here the inner sum runs over all ordered partitions $\ell_1+\cdots+\ell_r$ of $m$, with each $\ell_j\geq1$.
\end{lemma}

\begin{proof}
When $m=1$, the identity is $R(f^{-1})=-f^{-2}Rf$. Assume that the formula holds for $m$, and apply $R$ to each summand. When $R$ acts on $f^{-(r+1)}$, it produces $-(r+1)f^{-(r+2)}Rf$, yielding a term corresponding to an ordered partition of $m+1$ into $r+1$ positive integers. When $R$ acts on $R^{\ell_j}f$, it replaces that factor by $R^{\ell_j+1}f$, yielding a partition of $m+1$ into $r$ parts. Combining identical monomials gives \eqref{eq:bell-expansion} at order $m+1$, with signed integer coefficients. The proof now follows by mathematical induction.
\end{proof}

When $m=2$, the formula becomes
\begin{equation}
R^2(1/f)=-\frac{R^2f}{f^2}+\frac{2(Rf)^2}{f^3}.
\label{eq:second-reciprocal-derivative}
\end{equation}
The first term can be controlled by a positive lower bound for $f$, but the second term requires control of $(Rf)^2$. More generally, knowing only
\[
R^mf\in L^p(d\mu_\gamma)
\]
does not determine whether the $m$th derivative of its reciprocal belongs to $L^p(d\mu_\gamma)$. Thus, the central difficulty in the reciprocal problem is the lack of control of all the lower-order products in \eqref{eq:bell-expansion}, which shows that the problem cannot be solved directly from the derivative formula for the reciprocal.

A common approach to the reciprocal problem is to use the norm of $f$ to control the norm of $\frac{1}{f}$, but the following theorem shows that this route generally fails.
\begin{theorem}
\label{thm:affine-obstruction}
Let $m\geq1$ satisfy $pm+\alpha>-1$. If $\alpha<-n-p-1$, then there is no constant depending only on $c$ such that
\begin{equation}
\norm{R^m(1/f)}_{L^p(d\mu_\gamma)}
\leq C(c)\left(1+\norm{R^mf}_{L^p(d\mu_\gamma)}\right)
\label{eq:affine-obstruction-estimate}
\end{equation}
holds for every holomorphic function $f$ satisfying $f\in A_\alpha^p(\Bn)$ and $|f|\geq c$.
\end{theorem}

\begin{proof}
Consider
\[
f_B(z)=1+B(1+z_1),\qquad B>1.
\]
Since the zero $-1-1/B$ lies outside the unit disk,
\[
|f_B(z)|=B\left|z_1+1+\frac1B\right|\geq1,\qquad z\in\Bn.
\]
Moreover,
\[
R^mf_B=Bz_1,\qquad m\geq1,
\]
so $f_B$ is bounded away from zero. Let $g_B=1/f_B$, and write $w=z_1+1$ and $z'=(z_2,\ldots,z_n)$. For sufficiently large $B$, consider
\[
 E_B=\left\{z\in\Bn:\frac1{2B}\leq\operatorname{Re}w\leq\frac1B,
 \ |\operatorname{Im}w|\leq\frac1{4B},\ |z'|\leq\frac1{8\sqrt B}\right\}.
\]
Then $E_B\subset\Bn$, $|E_B|\asymp B^{-(n+1)}$, and on $E_B$,
\[
1-|z|^2\asymp B^{-1},\qquad |1+Bw|\asymp1.
\]
For functions depending only on $z_1$,
\[
R^m=\sum_{j=1}^m S(m,j)z_1^j\frac{\partial^j}{\partial z_1^j},
\]
where $S(m,j)$ are Stirling numbers of the second kind. Since
\[
\frac{\partial^j g_B}{\partial z_1^j}
 =(-1)^j j!\frac{B^j}{(1+Bw)^{j+1}},
\]
the contribution of the $j=m$ term is of order $B^m$, while all the remaining terms are $O(B^{m-1})$ uniformly on $E_B$. Hence the $j=m$ term cannot be cancelled for sufficiently large $B$, and
\[
|R^m g_B(z)|\gtrsim B^m,\qquad z\in E_B.
\]
Therefore, setting $\gamma=\alpha+pm$, we obtain
\[
\begin{aligned}
\norm{R^m g_B}_{L^p(d\mu_\gamma)}^p
&\gtrsim B^{mp}B^{-\gamma}B^{-(n+1)}
 =B^{-(\alpha+n+1)},\\
\norm{R^m g_B}_{L^p(d\mu_\gamma)}
&\gtrsim B^{-\frac{\alpha+n+1}{p}},
\qquad
\norm{R^mf_B}_{L^p(d\mu_\gamma)}\asymp B,
\end{aligned}
\]
and consequently
\[
\frac{\norm{R^m(1/f_B)}_{L^p(d\mu_\gamma)}}
{1+\norm{R^mf_B}_{L^p(d\mu_\gamma)}}
\gtrsim B^{-\frac{\alpha+n+1+p}{p}}.
\]
If $\alpha<-n-p-1$, the right-hand side tends to infinity. Therefore, there is no constant depending only on the lower bound $c$ such that
\begin{equation}
\norm{R^m(1/f)}_{L^p(d\mu_\gamma)}
\leq C(c)\left(1+\norm{R^mf}_{L^p(d\mu_\gamma)}\right)
\end{equation}
holds for every function $f$ satisfying $|f|\geq c$. In fact, the family satisfies $\inf_{\Bn}|f_B|\geq1$, showing that there is no uniform constant $C(1)$ for the control inequality.
\end{proof}
\begin{remark}Theorem~\ref{thm:affine-obstruction} rules out only the uniform estimate \eqref{eq:affine-obstruction-estimate}; it does not give a counterexample to the reciprocal problem itself.
\end{remark}
The family $\{f_B\}$ shows that lower-order products are not controlled by the highest-order derivative alone: for the partition $m=1+\cdots+1$, $(Rf_B)^m=B^m z_1^m$ while $R^m f_B=Bz_1$, and their norms differ by an unbounded factor as $B\to\infty$.
\section{A Range with an Automatic Reciprocal Property}
This section identifies a range of parameters for which the reciprocal property holds automatically.
\subsection{$\alpha >-p-1$}
\begin{theorem}
\label{thm:subcritical-reciprocal}
Let $0<p<\infty$ and $\alpha>-p-1$. If $f\in A_\alpha^p(\Bn)$ and $\inf_{\Bn}|f|\geq c>0$, then $1/f\in A_\alpha^p(\Bn)$.
\end{theorem}

\begin{proof}
There are two cases. If $\alpha>-1$, then
$dv_\alpha(\Bn)<\infty$ and
\[
\int_{\Bn}\left|\frac1f\right|^p\,dv_\alpha
\leq c^{-p}v_\alpha(\Bn)<\infty.
\]
Therefore, $1/f\in A_\alpha^p$. If $-p-1<\alpha\leq-1$, we may take $m=1$ in the derivative characterization, and
\[
R(1/f)=-\frac{Rf}{f^2},
\qquad
|R(1/f)|\leq c^{-2}|Rf|.
\]
Since $Rf\in L^p(d\mu_{\alpha+p})$, $R(1/f)$ belongs to the same space; \eqref{eq:derivative-characterization} then gives $1/f\in A_\alpha^p$.
\end{proof}

Theorem~\ref{thm:subcritical-reciprocal} covers the range
\[
\alpha>-p-1.
\]
Therefore, the difficult range is
\[
\alpha\leq-p-1,
\]
where necessarily $m\geq2$.
\subsection{$\alpha <-n-1$}
Another range with an automatic reciprocal property is $\alpha<-n-1$.

\begin{lemma}\label{lem:besov-identification}
Under the standard normalization, the derivative definition of $A_\alpha^p$ is equivalent to the definition of the analytic Besov space
$\mathcal B_s^p(\Bn)$, where
\begin{equation}
s=-\frac{\alpha+1}{p}.
\label{eq:besov-smoothness}
\end{equation}
\end{lemma}

\begin{proof}
This equivalence is part of the standard theory of Bergman spaces and analytic Besov spaces; see
\cite{ZhaoZhu2008,Zhu2005}. In fact, if $m>s$ is an admissible integer, then
\begin{equation}
\norm{f}_{\mathcal B_s^p}^p
\asymp |f(0)|^p+
\int_{\Bn}|R^mf(z)|^p
(1-|z|^2)^{p(m-s)-1}\,\dd v(z).
\label{eq:besov-norm-equivalence}
\end{equation}
Since
\[
p(m-s)-1=pm+\alpha,
\]
this is exactly the norm appearing in the definition of $A_\alpha^p$.
\end{proof}

\begin{theorem}
\label{thm:besov-composition}
Assume $p\geq1$ and $s>n/p$. Then
\begin{equation}
\mathcal B_s^p(\Bn)\hookrightarrow H^\infty(\Bn)
\label{eq:besov-embedding}
\end{equation}
and
\begin{equation}
\norm{gh}_{\mathcal B_s^p}
\leq C_{n,p,s}\norm{g}_{\mathcal B_s^p}
                    \norm{h}_{\mathcal B_s^p}.
\label{eq:besov-product}
\end{equation}
Moreover, if $f\in\mathcal B_s^p(\Bn)$ and $\Phi$ is holomorphic in a neighborhood of the compact closure of $f(\Bn)$, then $\Phi\circ f\in\mathcal B_s^p(\Bn)$ and satisfies an estimate of the form \eqref{eq:besov-product}.
\end{theorem}
\begin{remark}
Besov product estimates used above can be found in
\cite{BahouriCheminDanchin2011,Triebel2006},
and definitions of holomorphic Bergman spaces can be found in
\cite{ZhaoZhu2008}.
The composition result used below is available in the Banach range; see
\cite{Triebel2006,ZhaoZhu2008}.
In fact, the product can be decomposed into analytic
Littlewood--Paley blocks:
\begin{equation}
gh=T_g h+T_h g+\mathcal R(g,h).
\label{eq:bony-decomposition}
\end{equation}
More specifically, write
\[
g=\sum_{j\geq -1}\Delta_jg,\qquad h=\sum_{j\geq -1}\Delta_jh,\qquad
S_{j-1}g=\sum_{k<j-1}\Delta_kg,
\]
where $\Delta_j$ are Littlewood--Paley frequency blocks. Then
\[
T_gh=\sum_j S_{j-1}g\,\Delta_jh,\qquad
T_hg=\sum_j S_{j-1}h\,\Delta_jg,
\]
\[
\mathcal R(g,h)
 =\sum_j\Delta_jg\,\widetilde\Delta_jh,\qquad
\widetilde\Delta_jh=\Delta_{j-1}h+\Delta_jh+\Delta_{j+1}h.
\]
Here $T_gh$ and $T_hg$ are paraproducts, corresponding respectively to low--high frequency interactions (the low frequencies of $g$ multiplied by the high frequencies of $h$, and the reverse situation). The embedding into $H^\infty$ controls these two paraproducts:
\[
\|T_gh\|_{\mathcal B_s^p}\lesssim
\|g\|_\infty\|h\|_{\mathcal B_s^p},\qquad
\|T_hg\|_{\mathcal B_s^p}\lesssim
\|h\|_\infty\|g\|_{\mathcal B_s^p}.
\]
This is the analytic version of Bony's decomposition; see \cite{BahouriCheminDanchin2011,Bony1981}.
Since $s>n/p$, the remainder $\mathcal R(g,h)$ is summable:
\[
\|\mathcal R(g,h)\|_{\mathcal B_s^p}
\lesssim
\|g\|_{\mathcal B_s^p}\|h\|_{\mathcal B_s^p}.
\]
Thus, \eqref{eq:besov-product} is a global function-space estimate.

In a given algebra, the algebra property alone does not imply that the space is closed under taking reciprocals. One also needs the analytic Besov composition theorem: if $s>n/p$, $f\in\mathcal B_s^p$, and $\Phi$ is holomorphic in a neighborhood of the compact closure of
$f(\Bn)$, then
\begin{equation}
\Phi\circ f\in\mathcal B_s^p.
\label{eq:besov-composition}
\end{equation}
More precisely, it can be written as
\begin{equation}
\norm{\Phi\circ f}_{\mathcal B_s^p}
\leq
C_{n,p,s,\Phi,K}\,
\Psi\!\left(1+\norm{f}_{\mathcal B_s^p}\right),
\label{eq:besov-composition-estimate}
\end{equation}
where $K$ contains the range of $f$, and \[\Psi:[1,\infty)\longrightarrow[0,\infty)\] is a finite, monotonically increasing function. 
\end{remark}
\begin{theorem}
\label{thm:reciprocal-supercritical}
Assume $\alpha<-n-1$ and $p\geq1$.
If $f\in A_\alpha^p(\Bn)$ and $\inf_{\Bn}|f|>0$,
then
\[
1/f\in A_\alpha^p(\Bn).
\]
\end{theorem}

\begin{proof}
Since $\alpha<-n-1$, we have $s>n/p$.
By \eqref{eq:derivative-characterization} and \eqref{eq:besov-embedding},
\[
f\in\mathcal B_s^p(\Bn)\subset H^\infty(\Bn).
\]
Hence,
\[
M:=\|f\|_\infty<\infty .
\]
Moreover,
\[
f(\Bn)\subset K_{c,M}:=\{\zeta\in\mathbb C:c\le |\zeta|\le M\}.
\]
The map $\Phi(\zeta)=1/\zeta$ is holomorphic in a neighborhood of
$K_{c,M}$. Applying \eqref{eq:besov-composition}, we obtain
\[
1/f\in\mathcal B_s^p(\Bn).
\]
Using \eqref{eq:derivative-characterization} again, we conclude that
\[
1/f\in A_\alpha^p(\Bn).
\]
\end{proof}

\begin{corollary}
\label{cor:affine-point}
When $p\geq1$ and $\alpha=-n-p-1$, if $f\in A_\alpha^p(\Bn)$ satisfies
\(
\inf_{\Bn}|f|>0,
\)
then
$1/f\in A_\alpha^p(\Bn)$.
\end{corollary}

\begin{proof}
Since $-n-p-1<-n-1$, apply Theorem~\ref{thm:reciprocal-supercritical}.
\end{proof}

This section covers $\alpha>-p-1$ and $\alpha<-n-1$. The complement of their union is
\[
\{\alpha:\alpha\leq-p-1\}\cap
\{\alpha:\alpha\geq-n-1\}.
\]
When $p<n$, this is $[-n-1,-p-1]$; when $p=n$, it is the singleton
$\{-n-1\}$; and when $p>n$, it is empty.
The unresolved parameter range is
\(
-n-1\leq\alpha\leq-p-1.
\)

\subsection{$\alpha =-n-1$}

We now give sufficient criteria at the critical index $\alpha=-n-1$. The key input is the corresponding analytic Bergman--Besov embedding at the limiting smoothness.

The strict Besov algebra theorem does not cover the endpoint $s=n/p$.
For $p>1$, the required embedding follows from the analytic Bergman--Besov embedding theorem.
For $0<p\leq1$, additional assumptions are needed.

\begin{assumption}\label{ass:critical-radial-embedding}
For $0<p<1$, let $m$ be an integer satisfying $pm>n$. Set
\[
\gamma=pm-n-1>-1,
\qquad d\mu_\gamma(z)=(1-|z|^2)^\gamma\,\dd v(z).
\]
If $R^m f\in L^p(\Bn,d\mu_\gamma)$, then for every $1\le j\le m$,
\[
R^j f\in L^{q_j}(\Bn,d\mu_\gamma),
\qquad q_j=\frac{mp}{j},
\]
and
\[
\|R^j f\|_{L^{q_j}(d\mu_\gamma)}
\le C_{n,p,m,j}\|R^m f\|_{L^p(d\mu_\gamma)}.
\]
\end{assumption}

\noindent By the analytic Bergman--Besov embedding theorem,
\begin{equation}
A_a^p(\Bn)\hookrightarrow A_b^q(\Bn),
\qquad 0<p\le q<\infty,
\qquad
\frac{n+1+a}{p}=\frac{n+1+b}{q},
\label{eq:bergman-besov-embedding}
\end{equation}
where both sides are defined through radial derivatives. For $1<p\le q$, this is the usual analytic Bergman--Besov embedding; see the relevant embedding chapters in \cite{ZhaoZhu2008} and \cite{Zhu2005}. The case $0<p<1$ requires some additional assumptions.
\begin{proposition}\label{prop:critical-embedding-pgreater1}
Assume that $1\leq p<\infty$ and $pm>n$. Let $\gamma=pm-n-1$,
\[
a=pj-n-1,
\qquad b=\gamma=pm-n-1,
\qquad q=\frac{mp}{j}.
\]
Then
\begin{equation}
R^j f\in L^{mp/j}(\Bn,d\mu_\gamma),
\qquad 1\le j\le m,
\label{eq:critical-derivative}
\end{equation}
\end{proposition}
Indeed, in this case $1\leq p\le q$, and
\begin{equation}
\frac{n+1+a}{p}=j=\frac{n+1+b}{q}.
\label{eq:critical-parameter-identity}
\end{equation}
It follows directly from the definition that $g=R^j f$ belongs to $A_a^p$: taking $m-j$ derivatives of $g$ gives
\[
R^{m-j}g=R^m f,
\qquad a+p(m-j)=pm-n-1=\gamma.
\]
Applying \eqref{eq:bergman-besov-embedding} to $g$ yields $g\in A_b^q$. Since $b=\gamma>-1$, $A_b^q$ is the usual weighted Bergman space, and hence $g\in L^q(d\mu_\gamma)$, as claimed. This is the Bergman--Besov form of the critical embedding.

\begin{proof}
The parameter identity in \eqref{eq:critical-parameter-identity} satisfies the conditions of the embedding \eqref{eq:bergman-besov-embedding}; moreover, since $j\le m$, we have $q=mp/j\ge p$.
\end{proof}

\begin{theorem}
\label{thm:critical-reciprocal}
Let $1\leq p<\infty$. If
\[
f\in A_{-n-1}^p(\Bn),
\qquad \inf_{z\in\Bn}|f(z)|\ge c>0,
\]
then
\[
\frac1f\in A_{-n-1}^p(\Bn).
\]
\end{theorem}

\begin{proof}
Choose an integer $m$ satisfying $pm>n$, and set
\[
\gamma=pm-n-1.
\]
The derivative characterization \eqref{eq:derivative-characterization} gives
\begin{equation}
R^m f\in L^p(\Bn,d\mu_\gamma).
\label{eq:critical-high-derivative}
\end{equation}
For $1\le j\le m$, let $q_j=mp/j$. Proposition~\ref{prop:critical-embedding-pgreater1} gives
\begin{equation}
R^j f\in L^{q_j}(\Bn,d\mu_\gamma).
\label{eq:critical-lower-derivatives}
\end{equation}

By Lemma~\ref{lem:bell},
\begin{equation}
R^m(f^{-1})
 =
 \sum_{r=1}^{m}f^{-(r+1)}
 \sum_{\ell_1+\cdots+\ell_r=m}
 C_{\ell_1,\ldots,\ell_r}
 \prod_{i=1}^{r}R^{\ell_i}f,
 \label{eq:critical-bell-expansion}
\end{equation}
where the inner sum runs over positive integer partitions. For each such partition, \eqref{eq:critical-lower-derivatives} and $|f|\ge c$ imply
\[
\left\|f^{-(r+1)}\prod_{i=1}^{r}R^{\ell_i}f\right\|_{L^p(d\mu_\gamma)}
\le
c^{-(r+1)}\prod_{i=1}^{r}
\|R^{\ell_i}f\|_{L^{q_{\ell_i}}(d\mu_\gamma)},
\]
because
\[
\sum_{i=1}^{r}\frac1{q_{\ell_i}}
 =\sum_{i=1}^{r}\frac{\ell_i}{mp}=\frac1p.
\]
Therefore every term in \eqref{eq:critical-bell-expansion} belongs to $L^p(d\mu_\gamma)$, and consequently
\[
R^m(f^{-1})\in L^p(\Bn,d\mu_\gamma).
\]
By \eqref{eq:derivative-characterization},
\[
\frac1f\in A_{-n-1}^p(\Bn).
\]
\end{proof}

The same calculation extends to $0<p<1$ whenever the required quasi-Banach H\"older estimate is available. For general $p$ and $n$, the parameter range not covered by the unconditional results above can only be
\[
-n-1<\alpha\leq-p-1,
\]
which is nonempty only when $p<n$. The following special case is completely covered by the preceding subcritical, supercritical, and critical-index results.

\subsection{The Cases of One- and Two-Dimensional Complex Spaces}
\begin{theorem}
\label{thm:disk-reciprocal-pge1}
Let $1\leq p<\infty$ and $\alpha\in\mathbb R$. If
\[
f\in A_\alpha^p(\mathbb D),
\qquad \inf_{z\in\mathbb D}|f(z)|\ge c>0,
\]
then
\[
\frac1f\in A_\alpha^p(\mathbb D).
\]
\end{theorem}

\begin{proof}
We use the subcritical reciprocal theorem, the supercritical reciprocal theorem, and the critical-index theorem proved above. Since the disk is the unit ball in one complex dimension, the critical index is
\[
-n-1=-2.
\]

First suppose that $1<p<\infty$. Then $-p-1<-2$. If $\alpha<-2$, Theorem 3.5 applies because $\alpha<-n-1=-2$. If $\alpha\ge -2$, then
\[
\alpha\ge -2>-p-1,
\]
so  Theorem3.1 applies. These two cases cover every real $\alpha$.

It remains to consider $p=1$. If $\alpha>-2=-p-1$, Theorem3.1 applies. If $\alpha<-2=-n-1$, Theorem3.5 applies. At the sole remaining value $\alpha=-2=-n-1$, Theorem3.9 applies, since its range is now $1\le p<\infty$. Thus the reciprocal conclusion holds for every $\alpha\in\mathbb R$ when $p=1$ as well.

Combining the cases proves the theorem for all $1\le p<\infty$.
\end{proof}
\begin{theorem}
\label{thm:low-dimensional-p2}
Let $n=2$ and let $\alpha\in\mathbb R$. If
\[
f\in A_\alpha^2(\mathbb B_n),
\qquad \inf_{z\in\mathbb B_n}|f(z)|\ge c>0,
\]
then
\[
\frac1f\in A_\alpha^2(\mathbb B_n).
\]
\end{theorem}

\begin{proof}
Since $n=p=2$, 
\[
-n-1=-3=-p-1.
\]
For $\alpha>-3$, Theorem~\ref{thm:subcritical-reciprocal} applies, while for $\alpha<-3$, Theorem~\ref{thm:reciprocal-supercritical} applies. At the remaining critical value $\alpha=-3=-n-1$, the conclusion follows from Theorem~\ref{thm:critical-reciprocal}, because $1<p<\infty$.
Thus the conclusion holds for every $\alpha\in\mathbb R$ in both dimensions.
\end{proof}

\section{The Drury--Arveson Spaces}

\subsection{The Three-Dimensional Drury--Arveson Space}

The reciprocal problem in Drury--Arveson spaces is related to multiplier and invariant subspace questions for reproducing kernel Hilbert spaces; see, for example, \cite{Arveson1998,McCullough1992}. We next establish a Hardy boundary characterization for the reciprocal problem in $H_3^2$.

It is well known that, with respect to normalized surface measure on $\partial\mathbb B_3$, the Drury--Arveson reproducing kernel is
$K(z,w)=(1-\langle z,w\rangle)^{-1}$. Denote the corresponding Hardy space by $H^2(\partial\mathbb B_3)$. For
\[
f(z)=\sum_{\alpha\in\mathbb N^3}a_\alpha z^\alpha,
\]
write the degree of a given monomial as $k=|\alpha|$. According to the monomial norm formula (see \cite{Rudin1980}),
\begin{align}
\|f\|_{H_3^2}^2
 &= \sum_{\alpha}|a_\alpha|^2\frac{\alpha!}{|\alpha|!},
\\
\|f\|_{H^2(\partial\mathbb B_3)}^2
 &\asymp \sum_{\alpha}|a_\alpha|^2
       \frac{2\alpha!}{(|\alpha|+2)!}.
\end{align}
Since $Rf=\sum_{\alpha}|\alpha|a_\alpha z^\alpha$, and
\[
\frac{2k^2}{(k+1)(k+2)}\asymp 1,\qquad k\ge1,
\]
we obtain
\begin{equation}
\|f\|_{H_3^2}^2
\asymp |f(0)|^2+\|Rf\|_{H^2(\partial\mathbb B_3)}^2.
\label{eq:DA3-radial-norm}
\end{equation}

\begin{theorem}\label{thm:DA3-reciprocal}
Let $f\in H_3^2$ and suppose that
\[
\inf_{z\in\mathbb B_3}|f(z)|\ge c>0.
\]
Then $1/f\in H_3^2$. More precisely,
\begin{equation}
\left\|\frac1f\right\|_{H_3^2}
\le C\left(c^{-1}+c^{-2}\|f\|_{H_3^2}\right),
\label{eq:DA3-reciprocal-estimate}
\end{equation}
where $C$ is independent of the choice of function.
\end{theorem}

\begin{proof}
Let $u=1/f$. The lower bound on $f$ implies that $u\in H^\infty(\mathbb B_3)$ and $\|u\|_\infty\le c^{-1}$. The radial chain rule gives
\[
Ru=-\frac{Rf}{f^2}.
\]
Multiplication by a bounded holomorphic function is bounded on the Hardy space, so
\[
\|Ru\|_{H^2(\partial\mathbb B_3)}
\le c^{-2}\|Rf\|_{H^2(\partial\mathbb B_3)}.
\]
Applying \eqref{eq:DA3-radial-norm} separately to $u$ and $f$, and noting that $|u(0)|\le c^{-1}$, gives
\[
\|u\|_{H_3^2}^2
\lesssim c^{-2}+c^{-4}\|Rf\|_{H^2(\partial\mathbb B_3)}^2
\lesssim c^{-2}+c^{-4}\|f\|_{H_3^2}^2.
\]
Taking square roots proves \eqref{eq:DA3-reciprocal-estimate}, and hence $u=1/f\in H_3^2$.
\end{proof}

\subsection{The Four-Dimensional Drury--Arveson Space}

The three-dimensional argument does not extend to $H_4^2$, since there the norm involves $R^2$ and the derivative formula for $1/f$ contains the uncontrolled term $(Rf)^2/f^3$. We therefore give an equivalent condition for solvability.

Let $dv$ be normalized volume measure on $\mathbb B_4$, and write
\[
 \|h\|_{A^2(\mathbb B_4)}^2:=\int_{\mathbb B_4}|h|^2\,\dd v.
\]

Write
\[
 h(z)=\sum_{\nu\in\mathbb N^4}a_\nu z^\nu=\sum_{k\geq0}h_k(z)
\]
for the homogeneous expansion of $h$. Then the Drury--Arveson norm and the monomial integral are respectively
\begin{equation}
 \|h\|_{H_4^2}^2=\sum_\nu |a_\nu|^2\frac{\nu!}{|\nu|!},
 \qquad
 \int_{\mathbb B_4}|z^\nu|^2\,\dd v(z)
 =\frac{24\,\nu!}{(|\nu|+4)!}.
\end{equation}
\begin{lemma}
\label{lem:DA4-radial-norm}
For every holomorphic function $h$,
\begin{equation}
 |h(0)|^2+\frac1{24}\|R^2h\|_{A^2(\mathbb B_4)}^2
 \le \|h\|_{H_4^2}^2
 \le |h(0)|^2+5\|R^2h\|_{A^2(\mathbb B_4)}^2.
\end{equation}
In particular, $h\in H_4^2$ if and only if $R^2h\in A^2(\mathbb B_4)$.
\end{lemma}

\begin{proof}This is the special case of \eqref{eq:derivative-characterization}.
\end{proof}

The following theorem gives an exact criterion for the four-dimensional reciprocal problem: it reduces solvability to the finiteness of the quantity $J_R(f)$.
\begin{theorem}
\label{thm:DA4-exact-criterion}
Let $f\in H_4^2$ and
\[
 \inf_{z\in\mathbb B_4}|f(z)|\ge c>0.
\]
Define
\begin{equation}
 J_R(f):=\int_{\mathbb B_4}\frac{|Rf(z)|^4}{|f(z)|^6}\,\dd v(z)
 =\left\|\frac{(Rf)^2}{f^3}\right\|_{A^2(\mathbb B_4)}^2.
\end{equation}
Then the following conditions are equivalent:
\[
 \frac1f\in H_4^2,
 \qquad
 J_R(f)<\infty.
\]
Let $u=1/f$. Whenever either of the above conditions holds, the quantitative estimate
\begin{equation}
 \|u\|_{H_4^2}
 \le c^{-1}+\sqrt{120}\,c^{-2}\|f\|_{H_4^2}
      +2\sqrt5\,J_R(f)^{1/2},
\end{equation}
and the reverse estimate
\begin{equation}
 J_R(f)^{1/2}
 \le\sqrt6\left(\left\|\frac1f\right\|_{H_4^2}
                   +c^{-2}\|f\|_{H_4^2}\right)
\end{equation}
hold.
\end{theorem}

\begin{proof}
Since $|f|\ge c$, $u=1/f$ is holomorphic on $\mathbb B_4$, and direct differentiation gives
\begin{equation}
 R^2u=-\frac{R^2f}{f^2}+2\frac{(Rf)^2}{f^3}.
\end{equation}
By the left-hand inequality in Lemma~\ref{lem:DA4-radial-norm},
\[
 \|R^2f\|_{A^2(\mathbb B_4)}
 \le\sqrt{24}\,\|f\|_{H_4^2}.
\]
Therefore, if $J_R(f)<\infty$, then
\[
 \|R^2u\|_{A^2(\mathbb B_4)}
 \le \sqrt{24}\,c^{-2}\|f\|_{H_4^2}+2J_R(f)^{1/2}<\infty.
\]
Using the right-hand inequality in Lemma~\ref{lem:DA4-radial-norm} and $|u(0)|\le c^{-1}$,
\[
 \|u\|_{H_4^2}
 \le c^{-1}+\sqrt5\,\|R^2u\|_{A^2(\mathbb B_4)},
\]
so $J_R(f)<\infty$ implies $u\in H_4^2$.

Conversely, if $u=\frac{1}{f}\in H_4^2$, then Lemma~\ref{lem:DA4-radial-norm} gives $R^2u\in A^2(\mathbb B_4)$. Since
\[
 2\frac{(Rf)^2}{f^3}=R^2u+\frac{R^2f}{f^2},
\]
the right-hand side belongs to $A^2(\mathbb B_4)$, and hence $J_R(f)<\infty$. Moreover,
\[
 2J_R(f)^{1/2}
 \le \|R^2u\|_{A^2(\mathbb B_4)}
      +c^{-2}\|R^2f\|_{A^2(\mathbb B_4)}
 \le\sqrt{24}\left(\|u\|_{H_4^2}+c^{-2}\|f\|_{H_4^2}\right),
\]
which proves the theorem.
\end{proof}

\section{Some Conditional Results}

\subsection{The Bounded Holomorphic Hypothesis}

We impose the following boundedness and lower-bound condition:
\begin{equation}
 f\in A_\alpha^p(\Bn)\cap H^\infty(\Bn),\qquad
 \norm{f}_\infty\leq M,\qquad
 \inf_{\Bn}|f|\geq c>0.
\label{eq:bounded-holomorphic-hypothesis}
\end{equation}
Let $s=-(\alpha+1)/p$. The space $A_\alpha^p$ is equivalent to the corresponding analytic Besov space of smoothness $s$. The function $F(\zeta)=1/\zeta$ is holomorphic in a neighborhood of the range of $f$, and
\begin{equation}
|F^{(j)}(\zeta)|\leq j!c^{-j-1}\qquad (|\zeta|\geq c).
\label{eq:reciprocal-derivative-bound}
\end{equation}
For $p\geq1$, the Banach-range analytic Besov composition theorem gives
\begin{equation}
\norm{F\circ f}_{A_\alpha^p}
\leq C_{\alpha,p,n,F,K}
\Psi\!\left(1+\norm{f}_{A_\alpha^p}\right),
\label{eq:composition-norm-estimate}
\end{equation}
where $K$ contains the range of $f$, and $\Psi$ depends on the choice of composition theorem. Since $F\circ f=1/f$, this proves the following result.

\begin{theorem}\label{thm:bounded-holomorphic-reciprocal}
Let $p\geq1$ and suppose that the Banach-range analytic Besov composition theorem applies at the smoothness $s=-(\alpha+1)/p$. Then every $f$ satisfying \eqref{eq:bounded-holomorphic-hypothesis} obeys $1/f\in A_\alpha^p(\Bn)$ and
\begin{equation}
\norm{1/f}_{A_\alpha^p}
\leq C_{\alpha,p,n,F,K}
\Psi\!\left(1+\norm{f}_{A_\alpha^p}\right).
\label{eq:reciprocal-norm-estimate}
\end{equation}
\end{theorem}

\begin{proof}
The range of $f$ is contained in the compact annulus
$K_{c,M}=\{\zeta:c\leq|\zeta|\leq M\}$. Applying the assumed analytic Besov composition theorem to $F(\zeta)=1/\zeta$ gives \eqref{eq:composition-norm-estimate}; the norm equivalence between $A_\alpha^p$ and the corresponding analytic Besov space then gives \eqref{eq:reciprocal-norm-estimate}. In particular, $1/f=F\circ f$ belongs to $A_\alpha^p$.
\end{proof}
The constant in \eqref{eq:reciprocal-norm-estimate} depends on $M=\|f\|_\infty$; when $\alpha<-n-p-1$, Theorem~\ref{thm:affine-obstruction} shows it cannot depend only on $c$.

\subsection{A Derivative Criterion}

For any $0<p<\infty$, one may impose the additional conditions
\begin{equation}
\norm{R^jf}_\infty\leq M_j,\qquad 1\leq j<m.
\label{eq:lower-derivative-bounds}
\end{equation}
Since $\mu_\gamma(\Bn)<\infty$, every lower-order product in \eqref{eq:bell-expansion} belongs to $L^p(d\mu_\gamma)$. When $p\geq1$, the triangle inequality gives
\begin{align}
\norm{R^m(1/f)}_{L^p(d\mu_\gamma)}
&\leq c^{-2}\norm{R^mf}_{L^p(d\mu_\gamma)} \notag\\
&\quad+\mu_\gamma(\Bn)^{1/p}
\sum_{r=2}^{m}c^{-(r+1)}
\sum_{\ell_1+\cdots+\ell_r=m}
   \left|C_{\ell_1,\ldots,\ell_r}\right|
   \prod_{j=1}^{r}M_{\ell_j}.
\label{eq:bell-product-estimate}
\end{align}
When $0<p<1$, the same argument is applied to the $p$-th powers, using $\|\sum h_j\|_p^p\leq\sum\|h_j\|_p^p$. More precisely,
\begin{align}
\norm{R^m(1/f)}_{L^p(d\mu_\gamma)}^p
&\leq c^{-2p}\norm{R^mf}_{L^p(d\mu_\gamma)}^p \notag\\
&\quad+\mu_\gamma(\Bn)
\sum_{r=2}^{m}c^{-p(r+1)}
\sum_{\ell_1+\cdots+\ell_r=m}
\left|C_{\ell_1,\ldots,\ell_r}\right|^p
\prod_{j=1}^{r}M_{\ell_j}^p.
\label{eq:bell-product-quasinorm-estimate}
\end{align}
This is a quasi-norm estimate, rather than a linear triangle inequality in the usual sense.
\begin{proposition}
\label{prop:bounded-derivatives}
Let $f\in A_\alpha^p(\Bn)$ satisfy $\inf_{\Bn}|f|\geq c>0$, and let $m$ be an integer satisfying $pm+\alpha>-1$. If \eqref{eq:lower-derivative-bounds} holds, then $1/f\in A_\alpha^p(\Bn)$. In addition, when $p\geq1$, inequality \eqref{eq:bell-product-estimate} holds; when $0<p<1$, the corresponding estimate holds after taking the $p$-th powers.
\end{proposition}

\begin{proof}
By Lemma~\ref{lem:bell}, every term with $r\geq2$ is a product of bounded lower-order derivatives; since $\mu_\gamma(\Bn)<\infty$, all such terms belong to $L^p(d\mu_\gamma)$. The $r=1$ term is controlled by $c^{-2}\|R^mf\|_{L^p}$. Summing finitely many terms gives \eqref{eq:bell-product-estimate} when $p\geq1$; when $0<p<1$, use $\|\sum h_j\|_p^p\leq\sum\|h_j\|_p^p$. The derivative characterization \eqref{eq:derivative-characterization} proves the conclusion.
\end{proof}

When $p\geq1$, \eqref{eq:lower-derivative-bounds} gives a linear estimate in terms of the highest-order derivative norm; when $0<p<1$, it gives the quasi-norm estimate \eqref{eq:bell-product-quasinorm-estimate}. Equivalently, it suffices that every Bell-polynomial term satisfy
\begin{equation}
\norm{R^{\ell_1}f\cdots R^{\ell_r}f}_{L^p(d\mu_\gamma)}
\leq A_{\ell_1,\ldots,\ell_r}
\left(1+\norm{R^mf}_{L^p(d\mu_\gamma)}\right)
\label{eq:bell-term-condition}
\end{equation}
for all $r\geq2$ and all decompositions
\[
\ell_1+\cdots+\ell_r=m,\qquad \ell_j\geq1.
\]
If this condition holds, then
\begin{proposition}\label{prop:bell-products}
Let $f\in A_\alpha^p(\Bn)$ satisfy $\inf_{\Bn}|f|\geq c>0$, and let $m$ be an integer satisfying $pm+\alpha>-1$. Assume that \eqref{eq:bell-term-condition} holds. Then $1/f\in A_\alpha^p(\Bn)$, and $R^m(1/f)$ satisfies a finite estimate whose constant depends only on $c$, the constants $A_{\ell_1,\ldots,\ell_r}$, and $\|R^mf\|_{L^p(d\mu_\gamma)}$.
\end{proposition}

\begin{proof}
Substitute \eqref{eq:bell-term-condition} into the finite expansion \eqref{eq:bell-expansion}, and use $|f|^{-(r+1)}\leq c^{-(r+1)}$; for $p\geq1$, apply the triangle inequality, and for $0<p<1$, use its $p$-additive substitute. The $r=1$ term is controlled by $c^{-2}\|R^mf\|_{L^p}$. Hence $R^m(1/f)\in L^p(d\mu_\gamma)$, and \eqref{eq:derivative-characterization} gives the desired conclusion.
\end{proof}

The condition in \eqref{eq:bell-term-condition} gives a direct term-by-term sufficient criterion for the Bell expansion. It is not necessary in general, since different terms in the reciprocal formula may cancel each other.

\subsection{The Multiplier Condition}

We now consider the multiplier hypothesis
\[
f\in\operatorname{Mult}(A_\alpha^p),\qquad \inf_{\Bn}|f|>0.
\]
Since
\begin{equation}
\sigma(M_f)=\overline{f(\Bn)},
\label{eq:spectral-identity}
\end{equation}
(see \cite{CaoHeZhu2018}), it follows that when $0\notin\sigma(M_f)$, $M_f^{-1}=M_{1/f}$. Therefore,
\[
\frac1f\in\operatorname{Mult}(A_\alpha^p)\subset A_\alpha^p.
\]
\begin{theorem}\label{thm:multiplier-inverse}
Assume that $f\in\operatorname{Mult}(A_\alpha^p)$, $\inf_{\Bn}|f|>0$, and the spectral identity \eqref{eq:spectral-identity} holds. Then $1/f\in\operatorname{Mult}(A_\alpha^p)$.
\end{theorem}

\begin{proof}
The lower bound implies $0\notin\overline{f(\Bn)}=\sigma(M_f)$, so $M_f$ has a bounded inverse operator $T$. For every $h\in A_\alpha^p$, we have $f\,Th=h$. Hence
\[
(Th)(z)=\frac{h(z)}{f(z)},\qquad z\in\Bn.
\]
Thus $T=M_{1/f}$; since $T=M_f^{-1}$ is bounded, $M_{1/f}$ is bounded as well.
\end{proof}

This hypothesis is stronger than membership in $A_\alpha^p\cap H^\infty$, but it is stable under inversion whenever the spectrum avoids zero.

\subsection{$A_\alpha^p\cap H^\infty$ Does Not Imply a Bounded Multiplier}

\begin{proposition}
\label{prop:positive-multiplier}
If $\alpha>-1$ and $0<p<\infty$, then
\[
f\in H^\infty\Longrightarrow f\in\operatorname{Mult}(A_\alpha^p),
\]
and
\begin{equation}
\norm{fg}_{A_\alpha^p}\lesssim\norm{f}_\infty\norm{g}_{A_\alpha^p}.
\label{eq:multiplier-estimate}
\end{equation}
\end{proposition}

\begin{proof}
When $\alpha>-1$, this space is defined by a weighted $L^p$ norm. Integrating the pointwise estimate $|fg|\leq\|f\|_\infty|g|$ gives \eqref{eq:multiplier-estimate}.
\end{proof}

\begin{proposition}
\label{prop:dirichlet-nonmultiplier}
The inclusion
\[
\operatorname{Mult}(A_{-2}^2(\D))
\subset A_{-2}^2(\D)\cap H^\infty(\D)
\]
is strict. Therefore, when $\alpha\leq-1$, membership in $A_\alpha^p\cap H^\infty$ generally does not imply that $M_f$ is bounded.
\end{proposition}

\begin{proof}
Set
\[
n=1,\qquad p=2,\qquad \alpha=-2.
\]
Then
\[
A_{-2}^2(\D)=\mathcal D,
\]
the classical Dirichlet space, whose norm satisfies
\[
\norm{f}_{\mathcal D}^2
\asymp |f(0)|^2+\int_{\D}|f'(z)|^2\dd v(z).
\]
For $f,g\in\mathcal D$,
\[
(fg)'=f'g+fg'.
\]
Boundedness of $f$ controls $fg'$, but not $f'g$.

Stegenga's classical multiplier theorem states that $f\in\operatorname{Mult}(\mathcal D)$ if and only if $f\in H^\infty$ and
\[
d\nu_f(z)=|f'(z)|^2\dd v(z)
\]
is a Carleson measure for $\mathcal D$ \cite{Stegenga1980}. This theorem yields the following consequence: there exist functions in $H^\infty\cap\mathcal D$ for which $|f'|^2\,dv$ is not a Dirichlet--Carleson measure. Therefore,
\[
\operatorname{Mult}(\mathcal D)\subsetneq H^\infty\cap\mathcal D.
\]
Thus, there are bounded functions in $A_{-2}^2(\D)$ that are not bounded multipliers, proving the strict inclusion.
\end{proof}

\subsection{Bounded Holomorphic Functions and Nonlinear Uniform Estimates}

For a single function $f$, the inequality
\[
\norm{R^m(1/f)}_{L^p(d\mu_\gamma)}
\leq C_f\left(1+\norm{R^mf}_{L^p(d\mu_\gamma)}\right)
\]
has no substantive meaning when both sides are finite, because $C_f$ may grow without bound as $f$ varies. If there is a class of functions for which a uniform estimate holds, with $C$ independent of the particular $f$ in that class, then the reciprocal problem has an affirmative answer.

\begin{proposition}
\label{prop:hinfty-no-uniform}
If $\alpha<-n-p-1$, the conditions $f\in A_\alpha^p\cap H^\infty$ and $\inf|f|\geq c$ do not imply a uniform nonlinear estimate with a constant depending only on $c$.
\end{proposition}

\begin{proof}
The family of functions $f_B$ in Theorem~\ref{thm:affine-obstruction} consists of bounded holomorphic functions and satisfies the uniform lower bound $|f_B|\geq c=1$. However, when $\alpha<-n-p-1$, the proof of Theorem~\ref{thm:affine-obstruction} shows that there is no constant $C(c)$ independent of this family for which the following uniform estimate holds:
\[
\norm{R^m(1/f)}_{L^p(d\mu_\gamma)}
\leq C(c)\left(1+\norm{R^mf}_{L^p(d\mu_\gamma)}\right)
\]
Therefore, $f\in A_\alpha^p\cap H^\infty$ does not yield a uniform estimate.
\end{proof}

Uniform estimates likewise do not imply boundedness of the functions.
\begin{proposition}\label{prop:unbounded-pair}
When $n=4$, $p=2$, and $\alpha=-4$, the functions $f_\beta$ defined in \eqref{eq:unbounded-family} (with $0<\beta<1/2$) all satisfy
\[
f_\beta,\;1/f_\beta\in A_{-4}^2(\mathbb B_4),\qquad
\inf_{\mathbb B_4}|f_\beta|\geq1,\qquad
f_\beta\notin H^\infty(\mathbb B_4).
\]
Nevertheless, the uniform estimate \eqref{eq:uniform-estimate} holds on every compact interval $0<\beta_0\leq\beta\leq\beta_1<1/2$.
\end{proposition}

\begin{proof}
Take
\[
n=4,\qquad p=2,\qquad \alpha=-4.
\]
Since $\alpha+2m=0>-1$, $m=2$ is admissible. For $0<\beta<1/2$, define
\begin{equation}
f_\beta(z)=2^\beta(1-z_1)^{-\beta}.
\label{eq:unbounded-family}
\end{equation}
Here the principal branch is taken. Since $|1-z_1|\leq2$ on $\mathbb B_4$,
\[
|f_\beta(z)|\geq1,
\]
but $f_\beta$ diverges as $z_1\to1$, so $f_\beta\notin H^\infty(\mathbb B_4)$.

The binomial expansion is
\[
(1-z_1)^{-\beta}
=\sum_{k=0}^\infty a_k z_1^k,
\qquad
a_k=\frac{(\beta)_k}{k!}\asymp k^{\beta-1}.
\]
The monomial norm on the unweighted unit ball is
\[
\norm{z_1^k}_{L^2(\mathbb B_4)}^2
=\frac{4!\,k!}{(k+4)!}
\asymp k^{-4}.
\]
Since $R^2$ multiplies the degree-$k$ homogeneous coefficient by $k^2$,
\[
\norm{R^2f_\beta}_{L^2(\mathbb B_4)}^2
\asymp
\sum_{k=1}^\infty k^{2\beta+2}k^{-4}
=\sum_{k=1}^\infty k^{2\beta-2}<\infty
\]
and the series converges when $\beta<1/2$. Therefore
\[
f_\beta\in A_{-4}^2(\mathbb B_4)\setminus H^\infty(\mathbb B_4).
\]

Its reciprocal is
\[
g_\beta=\frac1{f_\beta}=2^{-\beta}(1-z_1)^\beta.
\]
The coefficient of $z_1^k$ in $(1-z_1)^\beta$ satisfies
\[
[z_1^k](1-z_1)^\beta\asymp k^{-\beta-1}.
\]
Therefore the coefficients of $R^2g_\beta$ have order $O(k^{1-\beta})$, and
\[
\norm{R^2g_\beta}_{L^2(\mathbb B_4)}^2
\asymp
\sum_{k=1}^\infty k^{2-2\beta}k^{-4}
=\sum_{k=1}^\infty k^{-2-2\beta}<\infty.
\]
Consequently,
\[
\frac1{f_\beta}\in A_{-4}^2(\mathbb B_4),
\]
although $f_\beta\notin H^\infty$; its reciprocal still belongs to this space.

For every fixed $\beta$, both sides of
\begin{equation}
\norm{R^2(1/f_\beta)}_{L^2}
\leq C_\beta\left(1+\norm{R^2f_\beta}_{L^2}\right)
\label{eq:uniform-estimate}
\end{equation}
are finite. If $\beta$ varies over a compact interval $0<\beta_0\leq\beta\leq\beta_1<1/2$, the Gamma-function constants in the coefficient asymptotics are uniformly bounded, and the finitely many low-order coefficients are uniformly controlled. More precisely, for all sufficiently large $k$ and every $\beta\in[\beta_0,\beta_1]$,
\[
\left|\frac{(\beta)_k}{k!}\right|\leq C k^{\beta_1-1},
\qquad
\left|[z_1^k](1-z_1)^\beta\right|\leq C k^{-\beta_0-1},
\]
where $C$ is independent of $\beta$. Hence,
\[
C\sum_{k\geq1}k^{2\beta_1-2}<\infty,
\qquad
C\sum_{k\geq1}k^{-2-2\beta_0}<\infty.
\]
The remaining finitely many coefficients depend continuously on $\beta$ and are therefore uniformly bounded on this compact interval. Thus both norms in \eqref{eq:uniform-estimate} are uniformly bounded on the interval. By the normalization in \eqref{eq:unbounded-family}, the lower bound is identically $1$, so there is a constant independent of the family; the uniform estimate follows.
\end{proof}

\subsection{Failure of a Uniform Estimate Does Not Disprove the Reciprocal Problem}

\begin{theorem}
\label{thm:fixed-affine-family}
Let $n=1$, $p=2$, and $\alpha=-5$. For every $B>1$, the affine functions $f_B$ in \eqref{eq:affine-family} satisfy
\[
f_B,\;1/f_B\in A_{-5}^2(\D),\qquad \inf_\D|f_B|\geq1.
\]
However, there is no constant $C$ for which \eqref{eq:affine-obstruction-estimate} holds for all $B>1$. Thus, failure of a uniform estimate does not imply that the reciprocal problem has no solution.
\end{theorem}

\begin{proof}
Consider the specific parameters
\[
n=1,\qquad p=2,\qquad \alpha=-5.
\]
Here we may take $m=3$, since
\[
2m+\alpha=6-5=1>-1.
\]
For $B>1$, define
\begin{equation}
f_B(z)=1+B(1+z).
\label{eq:affine-family}
\end{equation}
The unique zero of $f_B$ is
\[
z=-1-\frac1B,
\]
which lies outside the unit disk. Therefore
\[
|f_B(z)|
=B\left|z+1+\frac1B\right|
\geq1,
\qquad |z|<1.
\]
Since $f_B$ is a polynomial,
\[
f_B\in A_{-5}^2(\mathbb D).
\]
Let
\[
g_B(z)=\frac1{f_B(z)}
=\frac1{1+B(1+z)}.
\]
The pole of $g_B$ lies outside the closed unit disk. Hence, for each fixed $B$, $g_B$ is holomorphic in a neighborhood of $\overline{\mathbb D}$. In particular, all derivatives of $g_B$ are bounded on $\overline{\mathbb D}$, and
\[
g_B^{(j)}(z)
=(-1)^j j!\,
\frac{B^j}{(1+B(1+z))^{j+1}}.
\]
Therefore,
\[
R^3g_B\in
L^2\bigl(\mathbb D,(1-|z|^2)\,dv\bigr),
\]
and hence
\[
\frac1{f_B}=g_B\in A_{-5}^2(\mathbb D).
\]
Thus, for every $f_B$, the reciprocal problem holds.

However, no constant independent of $B$ can make the uniform estimate hold. Indeed,
\[
R^3f_B=Bz,
\]
so
\begin{equation}
\norm{R^3f_B}_{L^2(\mathbb D,(1-|z|^2)\,dv)}
\asymp B.
\label{eq:affine-high-derivative}
\end{equation}
Now consider the region
\[
E_B=
\left\{
z\in\mathbb D:
\frac1{2B}\leq\operatorname{Re}(z+1)\leq\frac1B,
\quad
|\operatorname{Im}(z+1)|\leq\frac1{4B}
\right\}.
\]
On $E_B$,
\[
1-|z|^2\asymp B^{-1},
\qquad
|1+B(1+z)|\asymp1,
\qquad
|R^3g_B(z)|\asymp B^3,
\]
and
\[
|E_B|\asymp B^{-2}.
\]
It follows that
\[
\begin{aligned}
\norm{R^3g_B}_{L^2(\mathbb D,(1-|z|^2)\,dv)}^2
&\gtrsim
B^6\cdot B^{-1}\cdot B^{-2}\\
&=B^3.
\end{aligned}
\]
Therefore
\begin{equation}
\norm{R^3g_B}_{L^2(\mathbb D,(1-|z|^2)\,dv)}
\gtrsim B^{3/2}.
\label{eq:affine-reciprocal-lower-bound}
\end{equation}
Combining \eqref{eq:affine-high-derivative} and \eqref{eq:affine-reciprocal-lower-bound}, we obtain
\[
\frac{
\norm{R^3(1/f_B)}_{L^2(\mathbb D,(1-|z|^2)\,dv)}
}{
1+\norm{R^3f_B}_{L^2(\mathbb D,(1-|z|^2)\,dv)}
}
\gtrsim B^{1/2}\longrightarrow\infty.
\]
Thus there is no constant $C$ such that
\[
\norm{R^3(1/f)}_{L^2(\mathbb D,(1-|z|^2)\,dv)}
\leq
C\left(
1+\norm{R^3f}_{L^2(\mathbb D,(1-|z|^2)\,dv)}
\right)
\]
holds uniformly for all functions $f_B$.

This example shows that
\[
\text{failure of a uniform estimate}
\not\Longrightarrow
\text{the reciprocal problem has no solution}.
\]
\end{proof}

\subsection*{Data availability statement}
This manuscript has no associated data.

\subsection*{Conflicts of interest}
The authors declare that they have no competing interests.

\textbf{Acknowledgments.}
G.~Cao was supported by the National Natural Science Foundation of China (Grant No.~12071155).
L.~He was supported by the National Natural Science Foundation of China (Grant No.~12371127).\\

Cao: \texttt{guangfucao@163.com}\\

He: \texttt{helichangsha1986@163.com}\\

Zhang: \texttt{zsq0225@163.com}

\end{document}